\documentclass[11pt,a4paper]{amsart}%
\usepackage{amscd}
\usepackage{amsmath}
\usepackage{graphicx}
\usepackage{amsfonts}
\usepackage{amssymb}
\newtheorem{thm}{Theorem}[section]
\newtheorem{lem}[thm]{Lemma}

\newtheorem{proposition}[thm]{Proposition}
\newtheorem{definition}[thm]{Definition}
\newtheorem{remark}[thm]{Remark}

\numberwithin{equation}{section}

\begin{document}
\title[Mass-transportation approach to a 1D model with nonlocal velocity]{A mass-transportation approach to a one dimensional fluid mechanics model with
nonlocal velocity}
\author[J. A. Carrillo]{Jos\'e A. Carrillo}
\address{ICREA and Departament de Matem\`atiques, Universitat Aut\`onoma de Barcelona,
08193 Bellaterra (Barcelona), Spain. \textit{On leave from:} Department of
Mathematics, Imperial College London, London SW7 2AZ, UK.}
\author[L. C. F. Ferreira]{Lucas C. F. Ferreira}
\address{Universidade Estadual de Campinas, Departamento de Matem\'{a}tica, CEP
13083-970, Campinas-SP, Brazil.}
\author[J. C. Precioso]{Juliana C. Precioso}
\address{Universidade Estadual Paulista, Departamento de Matem\'atica, S\~ao Jos\'e do
Rio Preto-SP, CEP: 15054-000, Brazil.}
\thanks{JAC was partially supported by the Ministerio de Ciencia e Innovaci\'on, grant
MTM2011-27739-C04-02, and by the Ag\`encia de Gesti\'o d'Ajuts Universitaris i
de Recerca-Generalitat de Catalunya, grant 2009-SGR-345.}
\thanks{LCFF was supported by FAPESP-SP, CNPQ and CAPES, Brazil. }
\thanks{JCP was supported by CAPES, grant BEX2872/05-6, Brazil. }

\maketitle

\begin{abstract}
We consider a one dimensional transport model with nonlocal
velocity given by the Hilbert transform and develop a global
well-posedness theory of probability measure solutions. Both the
viscous and non-viscous cases are analyzed. Both in original and
in self-similar variables, we express the corresponding equations
as gradient flows with respect to a free energy functional
including a singular logarithmic interaction potential. Existence,
uniqueness, self-similar asymptotic behavior and inviscid limit of
solutions are obtained in the space $\mathcal{P}_{2}(\mathbb{R})$
of probability measures with finite second moments, without any
smallness condition. Our results are based on the abstract
gradient flow theory developed in \cite{Ambrosio}. An important
byproduct of our results is that there is a unique, up to
invariance and translations, global in time self-similar solution
with initial data in $\mathcal{P}_{2}(\mathbb{R})$, which was
already obtained in \textrm{\cite{Deslippe,Biler-Karch}} by
different methods. Moreover, this self-similar solution attracts
all the dynamics in self-similar variables. The crucial
monotonicity property of the transport between measures in one
dimension allows to show that the singular logarithmic potential
energy is displacement convex. We also extend the results to
gradient flow equations with negative power-law locally integrable
interaction potentials.
\end{abstract}



\section{Introduction}

In this work, we are interested in developing a well-posedness theory of
measure solutions to the equation
\begin{equation}
\left\{
\begin{array}
[c]{lcr}%
u_{t}+\left(  H(u)u\right)  _{x}=0 &  & \\
u(x,0)=u_{0}(x) &  &
\end{array}
\right.  \,, \label{eqprinc}%
\end{equation}
with general nonnegative initial Borel measures $u_{0}$. Here, the term $H(u)$
denotes the classical Hilbert transform
\[
H(u)=\frac{1}{\pi}P.V.\int_{\mathbb{R}}\frac{u(z)}{x-z}\,dz\,.
\]
Since the equation is of transport nature and in divergence form, we expect
sign preservation and mass conservation. Therefore, we will restrict our
attention to probability measures as initial data. This equation is nothing
else than a $1D$-dimensional continuity equation in which the velocity field
is given by the Hilbert transform and it has been proposed as a simplified
model in fluid mechanics \cite{Chae-Cordoba-Fontelos} and in dislocation
dynamics \cite{Biler-Karch} as we will discuss in next subsection. This
equation can be formally considered a particular example of the theory of
gradient flows in the space of probability measures \cite{Ambrosio} as it will
be further elaborated in subsection 1.2.

The main aim of this work is to show that unique measure solutions of
gradient-flow type can be constructed for the problem
\begin{equation}
\left\{
\begin{array}
[c]{lcr}%
u_{t}+\left(  H(u)u\right)  _{x}=\kappa u_{xx} &  & \\
u(x,0)=u_{0}(x) &  &
\end{array}
\right.  \, , \label{eq3}%
\end{equation}
with $\kappa\ge0$ and $u_{0}\in\mathcal{P}_{2}(\mathbb{R})$ the set of
probability measures on the real line with finite second moments. Moreover,
the solutions will continuously depend on both the initial data $u_{0}$ and
the viscosity parameter $\kappa\ge0$. The main tools of this construction are
the variational schemes based on optimal transportation theory originated for
the seminal work \cite{Jordan}.

Moreover, we will be able to characterize the large time behavior
of the solutions. In fact, we show that suitable scaled equations
related to \eqref{eqprinc} and \eqref{eq3} have unique stationary
solutions fixed by the normalization of the mass. Furthermore, we
are able to show that the solutions constructed converge for large
times to these stationary solutions exponentially fast in some
transport distance. These stationary solutions correspond to
self-similar solutions for the original equations.

This manuscript is organized as follows. In the next two subsections we recall
the main results already obtained in the literature and the origin of these
models. On the other hand, we introduce some basic notations and definitions
about optimal mass transportation theory essential to our construction.
Section 2 is devoted to introduce self-similar variables and rewrite our
problem in the form of a gradient flow in the space of probability measures of
a free energy functional. Key properties of the functionals involved are shown
in subsections 2.1 and 2.2. Finally, we state and prove our existence,
asymptotic behavior and inviscid limit results in Section 3.

\subsection{Motivation: Fluid and Fracture Mechanics}

One of the motivations to analyze these equations arose from the mathematical
fluid mechanics literature. In fact, it appears as a simplified one
dimensional model \cite{Baker-Morlet} mimicking the structure of the
$3D$-Navier-Stokes equations and the $2D$-quasi-geostrophic equations
\cite{Chae-Cordoba-Fontelos}:
\begin{equation}
\left\{
\begin{array}
[c]{lcr}%
u_{t}+(\theta\cdot\nabla)u=0 &  & \\
\theta=\nabla^{\bot}\phi,\hspace{0.2cm}u=-(-\triangle)^{\frac{1}{2}}\phi &  &
\\
u(x,0)=u_{0}(x), &  &
\end{array}
\right.  \label{eq1}%
\end{equation}
where $\nabla^{\bot}=(-\partial_{2},\partial_{1})$ and $u(x,t)$ represents the
air temperature. Since $\mbox{div}(R^{\bot}u)=0,$ rewriting the system
(\ref{eq1}) in terms of the Riesz transform given by \ \
\[
R_{j}(u)(x,t)=\frac{1}{2\pi}P.V.\int_{\mathbb{R}^{2}} \frac{(x_{j}-y_{j}%
)}{|x-y|^{3}}\,u(y,t)\,dy
\]
the result is
\begin{equation}
\left\{
\begin{array}
[c]{lcr}%
u_{t}+\mbox{div}\left[  \left(  R^{\bot}u\right)  u\right]  =0 &  & \\
u(x,0)=u_{0}(x), &  &
\end{array}
\right.  \label{eq2}%
\end{equation}
Considering $u(x,t)$ with $(x,t)\in[0,\infty)\times\mathbb{R}$ and
replacing the Riesz transform in (\ref{eq2}) by the Hilbert
transform in one dimension leads to \eqref{eqprinc} or \eqref{eq3}
with diffusion, see \cite{CLM} for previous related works and
simplified models. Mathematical fluid mechanics arguments have
been used to analyze existence and uniqueness, finite time blow-up
of smooth solutions, and other issues, see
\cite{Morlet,Chae-Cordoba-Fontelos,Cordoba-Fontelos1,Cordoba-Fontelos2,DongH,Li-Rodrigo}
and the references therein related to these equations and other
nonconservative variants.

More precisely in our case, sign-changing periodic $C^{1}$-solutions of
\eqref{eqprinc} blow up in finite time, in the sense that its $C^{1}$-norm
diverges in finite time as shown in \cite{Chae-Cordoba-Fontelos}. On the other
hand, global existence and uniqueness of smooth solutions for the Cauchy
problem on the whole real line is proved in \cite{Castro-Cordoba} for strictly
positive initial data for \eqref{eqprinc} and for general nonnegative initial
data for \eqref{eq3}. The same authors show that for non negative
touching-down initial data the Cauchy problem for \eqref{eqprinc} is locally
well-posed for smooth solutions and that solutions do blow up in finite time
in the $C^{1}$ norm.

Apart from the structural similarities, the equations \eqref{eq2} and
\eqref{eqprinc} have different properties. For instance, while the first one
has a Hamiltonian structure, the second one being one dimensional can be
considered rather as a gradient flow as we will discuss in next subsection.

The other source of motivation to analyze equations
\eqref{eqprinc} and \eqref{eq3} comes from dislocation dynamics in
crystals \cite{Head1,Head2,Head3,Deslippe}. Here, the unknown $u$
represents the number density of fractures per unit length in the
material. The existence of explicit self-similar solutions and the
convergence towards them was studied in
\cite{Deslippe,Biler-Karch} showing that nonnegative solutions
play an important role in the large time asymptotics of
\eqref{eq3} and related problems. In fact, we will give a
characterization of the self-similar solution as the minimizer of
a free energy functional intimately related to its gradient flow
structure. In this way, we will show that the solution does really
converge in suitable scaling and in transport distances to the
self-similar profile.

\subsection{Gradient Flows for Probability Measures}

Let us remind some basic facts about optimal mass transport, which will be
useful to our study of solutions of the Cauchy problem (\ref{eq3}). For more
details we refer the reader to \cite{Villani,Ambrosio}. Let us denote by
$\mathcal{P}(\mathbb{R}^{d})$ the space of probability measures on
$\mathbb{R}^{d}.$ We start reminding the definition of push forward of a
measure $\rho\in\mathcal{P}(\mathbb{R}^{d})$.

\begin{definition}
Let $\rho$ be a probability measure on $\mathbb{R}^{d}$ and let $T:\mathbb{R}%
^{d} \rightarrow\mathbb{R}$ be a Borel map. The push forward $T_{\sharp}%
\rho\in\mathcal{P}(\mathbb{R}^{d})$ of $\rho$ through $T$ is defined by
$T_{\sharp} \rho(I):=\rho(T^{-1}(I))$ for any Borel subset $I\subset
\mathbb{R}$. The measures $\rho$ and $T_{\sharp}\rho$ satisfies
\[
\int_{\mathbb{R}^{d}}f(T(x))d\rho(x)=\int_{\mathbb{R}^{d}}f(y)dT_{\sharp}%
\rho(y)\,,
\]
for every bounded or positive continuous function $f$.
\end{definition}

Let $\rho,\mu\in\mathcal{P}(\mathbb{R}^{d})$ and $T_{\rho}^{\mu}%
:\mathbb{R}^{d} \rightarrow\mathbb{R}^{d}$ such that $T_{\rho\text{ }\sharp
}^{\mu}\rho=\mu$. The map $T_{\rho}^{\mu}$ is called a transport map between
the probability measure $\rho$ and $\mu$. We also recall the notion of
transport plan between two probability measures.

\begin{definition}
Given two measures $\rho$ and $\mu$ of $\mathcal{P}(\mathbb{R}^{d})$ the set
of transport plans between them is defined by
\[
\Gamma(\rho,\mu):=\left\{  \gamma\in\mathcal{P}(\mathbb{R}\times
\mathbb{R}):\pi_{\sharp}^{1}\gamma=\rho,\pi_{\sharp}^{2}\gamma=\mu\right\}  ,
\]
where $\pi^{i}:\mathbb{R}^{d}\times\mathbb{R}^{d}\rightarrow\mathbb{R}^{d},$
$i=1,2$ are the projections onto the first and second coordinate: $\pi
^{1}(x,y)=x$, $\pi^{2}(x,y)=y$. In other words, transport plans are those
having marginals $\rho$ and $\mu.$
\end{definition}

Our aim is to study solutions of the Cauchy problem \eqref{eqprinc} and
(\ref{eq3}) in an appropriate subspace of $\mathcal{P}(\mathbb{R})$ endowed
with a transport distance, the so-called euclidean Wasserstein distance.
Consider the set
\[
\mathcal{P}_{2}(\mathbb{R}^{d})=\left\{  \rho\in\mathcal{P}(\mathbb{R}^{d}
):\int_{\mathbb{R}^{d}}\left\vert x\right\vert ^{2}d\rho(x)<\infty\right\}  \,
.
\]
The euclidean Wasserstein distance is defined on $\mathcal{P}_{2}%
(\mathbb{R}^{d})$ as:

\begin{definition}
For any probability measure $\rho, \mu\in\mathcal{P}_{2} (\mathbb{R}^{d})$ the
euclidean Wasserstein distance between them is defined by
\[
d_{2} (\rho, \mu):= \min\left\{  \left(  \int_{\mathbb{R}^{d} \times
\mathbb{R}^{d}} |y-x|^{2} d \gamma(x,y)\right)  ^{\frac{1}{2}}: \gamma
\in\Gamma(\rho, \mu)\right\}  .
\]

\end{definition}

We denote by $\Gamma_{0}(\rho,\mu)$ the set of optimal plans, i.e., the subset
of $\Gamma(\rho,\mu)$ where the minimum is attained, i.e,
\[
\Gamma_{0}(\rho,\mu)=\left\{  \gamma\in\Gamma(\rho,\mu):\int_{\mathbb{R}^{d}
\times\mathbb{R}^{d}}|y-x|^{2}d\gamma(x,y)=d_{2}^{2}(\rho,\mu)\right\}  .
\]
The space $\mathcal{P}_{2} (\mathbb{R}^{d})$ endowed with $d_{2}$ becomes a
complete metric space. The convergence in $d_{2}$ is equivalent to weak-$*$
convergence as measures together with convergence of the second moments, see
\cite[Theorem 7.12]{Villani}. We will denote by $\mathcal{P}_{2}%
^{ac}(\mathbb{R}^{d})$ the subset of probability measures with absolutely
continuous densities with respect to Lebesgue measure and finite second
moments. It is well-known that for any $\rho,\mu\in\mathcal{P}_{2}%
^{ac}(\mathbb{R}^{d})$, the minimum in the definition of $d_{2}$ is achieved
by a plan defined by an optimal map, i.e., by a plan defined by $\gamma
=(1_{\mathbb{R}^{d}}\times T_{\rho}^{\mu})_{\#}\rho$.

Let us remark that the Wasserstein distance in one dimension can be easily
characterized since the optimal transport map, if exists, in one dimension is
always a monotone non decreasing function. In fact, as shown in \cite[Theorem
2.18]{Villani}, the optimal plan in one dimension is independent of the cost
and given in terms of the distribution functions associated to the probability
measures and their pseudo-inverses. In fact, one can show that:

\begin{lem}
\label{tnondec} Given $\rho,\mu\in\mathcal{P}_{2}^{ac}(\mathbb{R})$, the
optimal transport map $T_{\rho}^{\mu}$ in $\mathbb{R}$ for $d_{2}$ between
them is essentially increasing, i.e., $T_{\rho}^{\mu}$ is increasing except in
a $\rho$-null set.
\end{lem}

\begin{proof}
Indeed, we can use the distribution function $F_{\rho}(x):=\rho((-\infty,x))$,
and define its pseudo-inverse of $F_{\rho}$ by the formula
\[
F_{\rho}^{-1}(s):=\sup\left\{  x\in\mathbb{R}:F_{\rho}(x)\leq s\right\}
,\hspace{0.3cm}\text{for }s\in\lbrack0,1].
\]
In one dimension, the optimal map for $d_{2}$, and for general costs, is given
by $T_{\rho}^{\mu}=F_{\mu}^{-1}\circ F_{\rho}$ satisfying obviously $T_{\rho
}^{\mu}(s_{1})\leq T_{\rho}^{\mu}(s_{2})$ for all $0\leq s_{1}\leq s_{2}\le1$.
Thus, the optimal transport map $T_{\rho}^{\mu}$ is nondecreasing. Since
$\rho\in\mathcal{P}^{ac}(\mathbb{R})$, $T_{\rho}^{\mu}$ is an injective
function except in a $\rho$-null set (see \cite[ Remark 6.2.11]{Ambrosio}).
Therefore it follows at once that $T_{\rho}^{\mu}$ is increasing except in a
$\rho$-null set, i.e., it is essentially increasing.
\end{proof}

Following the seminal ideas for the porous medium equation in \cite{Otto} and
the linear Fokker-Planck equation in \cite{Jordan}, a theory of gradient flows
in the space of probability measures $(\mathcal{P}_{2} (\mathbb{R}^{d}%
),d_{2})$ has been fruitfully applied to general class of
equations in the last decade
\cite{Ambrosio,Carrillo1,Carrillo,Agueh}. These equations are
continuity equations where the velocity field is given by the
gradient of the variational derivative of an energy functional.
More precisely, they are of the form
\begin{equation}
\label{generalequations}\frac{\partial\rho}{\partial t}=\mbox{div }\left(
\rho\nabla\frac{\delta{\mathcal{E}}}{\delta\rho}\right)  , \qquad
\mbox{in}\qquad(0,+\infty)\times\mathbb{R}^{d}\,,
\end{equation}
where the free energy functional $\mathcal{E}$ is given by
\begin{equation}
\label{generalfunctionals}{\mathcal{E}}[\rho]:=\int_{\mathbb{R}^{d}}
U(\rho(x))\, dx + \int_{\mathbb{R}^{d}} \rho(x)\, V(x)\, dx + \frac12
\iint_{\mathbb{R}^{d}\times\mathbb{R}^{d}} W(x-y)\, \rho(x)\,\rho(y)\, dx \,
dy\;
\end{equation}
under the basic assumptions $U:\mathbb{R}^{+}\to\mathbb{R}$ is a density of
internal energy, $V:\mathbb{R}^{d}\to\mathbb{R}$ is a confinement potential
and $W:\mathbb{R}^{d}\to\mathbb{R}$ is a symmetric interaction potential. The
internal energy $U$ should satisfy the following dilation condition,
introduced in McCann~\cite{McCann}
\begin{equation}
\label{condU}\lambda\longmapsto\lambda^{d} U(\lambda^{-d}) \qquad\text{is
convex non-increasing on $\mathbb{R}^{+}$}.
\end{equation}
The most classical case of application, as it is for our case, is $U(s) =
s\log s$, which identifies the internal energy with Boltzmann's entropy.

We can check that, at least formally, our equations of interest
\eqref{eqprinc} and \eqref{eq3} are of the form \eqref{generalequations} in
$d=1$ defined by the functional \eqref{generalfunctionals} with the choices:
$U=V=0$, and $W(x)=-\frac1\pi\log|x|$; and $U(s) = \kappa\,s\log s$, $V=0$,
and $W(x)=-\frac1\pi\log|x|$, respectively. However, the theory developed in
\cite{Ambrosio} is not directly applicable to \eqref{eqprinc} and \eqref{eq3}
for two reasons: this theory uses a convexity property of the functional
$\mathcal{E}$ that we will discuss next and the potentials $V,W$ have to be
smooth functions while we deal with the singular potential $-\frac1\pi\log|x|$
with an apriori unclear convexity properties.

The needed notion of convexity for functionals on measures was introduced in
\cite{McCann} and named displacement convexity. This notion provides
functionals of the form \eqref{generalfunctionals} with a natural convexity
structure allowing to show that the variational scheme introduced in
\cite{Jordan} is convergent under smoothness and convexity assumptions of the
confining and interaction potentials $V,W$ and the convexity property of the
internal energy in \eqref{condU}, see \cite{Ambrosio} for precise statements.
Let us define
\begin{equation}
\label{deflog}\tilde W(x)=
\begin{cases}
-\frac1\pi\log|x| & \mbox{ for } x\ne0\\
+\infty & \mbox{ at } x=0
\end{cases}
\, .
\end{equation}
In our case, we will show that the interaction functional $\mathcal{E}$ with
$U=V=0$ and $W=\tilde W$ given by \eqref{deflog} in one dimension is indeed
displacement convex. The intuition behind this is that $\tilde W(x)$ is
clearly convex for $x\ge0$ with the definition above and the optimal map
between two measures is essentially increasing as shown in Lemma
\ref{tnondec}. Therefore, when transporting measures we only ``see'' the
convex part of $\tilde W(x)$.

On the other hand, this convexity will allow us to avoid the singularity too.
In plain words, we will show that this interaction potential is extremely
repulsive in one dimension producing that any initial measure is
instantaneously regularized to an absolutely continuous measure for all $t>0$.
This behavior is very interesting compared to fully attractive potentials. In
fact, equation \eqref{generalequations} has been studied in $d=1$ with the
displacement concave attractive potential $W(x)=\frac1\pi\log|x|$, $U(s) =
\kappa\,s\log s$ and $V=0$, the so-called one dimensional version of the
Patlak-Keller-Segel model, in \cite{Carrillo2}. There, it is shown that the
variational scheme in \cite{Jordan} converges to a weak solution of the
equation in case the diffusion $\kappa$ does not go below certain critical
value. Let us point out that nonpositive solutions to \eqref{eq3} corresponds
easily to nonnegative solutions to this 1D-PKS model via reflection, cf.
\cite{Castro-Cordoba}.

Other related works for fully attractive potentials may lead to finite time
blow-up, in the sense of finite time aggregation to Delta Dirac points, see
\cite{BCL,CDFLS}. For fully repulsive potentials like the one we consider
here, we are aware about the recent work in \cite{BLL} dealing with the
asymptotic behavior of $L^{1}\cap L^{\infty}$-solutions in dimensions $d\ge2$
for the Newtonian potential from a more fluid mechanics perspective.

\section{Free Energy Properties}

With the purpose in mind to give a well-posedness theory for probability
measure solutions to \eqref{eqprinc} and \eqref{eq3}, we should keep in mind
that we are also interested in the asymptotic behaviour of the solutions. For
both reasons, it is obvious that a deep preliminary study of the minimization
and convexity properties of the free energy functionals involved has to be performed.

In order to find self-similar solutions to \eqref{eq3}, we will need to
rescale variables, as usually done in nonlinear diffusion equations \cite{CT}
to translate possible self-similar solutions onto stationary solutions. The
rescaled equations can also be considered gradient flows of certain free
energy functionals which are uniformly 1-convex functionals in the sense of
displacement convexity. These are the objectives of this section.

\subsection{Self-similar variables and gradient flow structure}

We introduce the following self-similar variables
\begin{equation}
\left\{
\begin{array}
[c]{lcr}%
y=x(1+2t)^{-\frac{1}{2}},\hspace{0.5cm}\mbox{for all }t>0\hspace
{0.2cm}\mbox{and}\hspace{0.2cm}x\in\mathbb{R} &  & \\
\tau=\dfrac{1}{2}\log(1+2t) &  &
\end{array}
\right.  \label{change}%
\end{equation}
Now, observe that if $u(x,t)$ is a solution of system (\ref{eq3}), then the
function
\begin{equation*}
\rho(y,\tau)=(2t+1)^{\frac{1}{2}}u(x,t) 
\end{equation*}
with $(y,\tau)$ defined by \eqref{change} satisfies the equation
\begin{equation}
\partial_{\tau}\rho=\partial_{y}(y\rho)+\kappa\partial_{yy}\rho-\left(
H(\rho)\rho\right)  _{y}. \label{eq-resc}%
\end{equation}
Then, we can write the system (\ref{eq3}) in the new variables as
\begin{equation}
\left\{
\begin{array}
[c]{lcr}%
\partial_{\tau}\rho=\partial_{y}(y\rho)+\kappa\partial_{yy}\rho-\left(
H(\rho)\rho\right)  _{y},\hspace{0.5cm} & \forall\tau>0\mbox{ and }y\in
\mathbb{R} & \\[2mm]%
\rho(y,0)=u_{0} & \forall y\in\mathbb{R} &
\end{array}
\right.  \,. \label{eq4.1}%
\end{equation}

Equation \eqref{eq4.1} has a gradient flow structure in the sense of
subsection 1.2, i.e, we can rewrite it in the following form
\begin{equation}
\frac{\partial\rho}{\partial t}=\frac{\partial}{\partial y}\left[  \rho
\frac{\partial}{\partial y}\left(  \kappa\log\rho-\frac{1}{\pi}\log|y|\ast
\rho+\frac{y^{2}}{2}\right)  \right]  ,\label{eq5.1}%
\end{equation}
where we have replaced the letter $\tau$ by $t$ again. {F}rom now
on, we identify the time dependent probability measure
$\rho(\cdot,t)=\rho_{t}$ with its
density with respect to Lebesgue and we use the notation $d\rho_{t}%
=d\rho(x,t)=\rho(x,t)\,dx$.

Let us begin by introducing a precise definition of a free energy functional
$E_{\kappa,\alpha}$ on the space of probability measures $\mathcal{P}%
_{2}(\mathbb{R})$. We define $E_{\kappa,\alpha}:\mathcal{P}_{2}(\mathbb{R}%
)\rightarrow$ $\mathbb{R\cup\{+\infty\}}$ as
\begin{equation}
E_{\kappa,\alpha}[\rho]=%
\begin{cases}
\kappa\,\mathcal{U}[\rho]+\alpha\,\mathcal{V}[\rho]+\mathcal{W}[\rho] &
\mbox{for }\rho\in\mathcal{P}_{2}^{ac}(\mathbb{R})\\[2mm]%
+\infty & \mbox{otherwise}
\end{cases}
, \label{func3}%
\end{equation}
with $\alpha,\kappa\geq0$ and where for $\rho\in\mathcal{P}_{2}(\mathbb{R})$
\begin{align*}
&  \mathcal{U}[\rho]:=%
\begin{cases}
\displaystyle\int_{\mathbb{R}}\rho(x)\log\rho(x)\,dx & \mbox{for
}\rho\in\mathcal{P}_{2}^{ac}(\mathbb{R})\\[2mm]%
+\infty & \mbox{otherwise}
\end{cases}
\,,\qquad\mathcal{V}[\rho]:=\int_{\mathbb{R}}\frac{x^{2}}{2}\rho(x)\,dx\,,\\
&  \mbox{and}\qquad\mathcal{W}[\rho]:=\frac{1}{2}\int_{\mathbb{R}^{2}}%
\tilde{W}(x-y)\rho(x)\rho(y)\,dx\,dy\,,
\end{align*}
where $\tilde{W}$ is given by (\ref{deflog}). We can now identify that
\eqref{eq4.1} or \eqref{eq5.1}, \eqref{eqprinc} and \eqref{eq3} belong to the
class of equations \eqref{generalequations} with the choices $W(x)=\tilde
{W}(x)$, $U(s)=\kappa\,s\log s$ and $V=\alpha\frac{x^{2}}{2}$ with different
values for $\kappa$ and $\alpha$. Thus, they are formal gradient flows of the
corresponding free energy functionals $E_{\kappa,\alpha}[\rho]$. It can be
easily checked that the functional $E_{\kappa,\alpha}$ is formally a Lyapunov
functional for the equation (\ref{eq-resc}), i.e,
\[
\frac{d}{dt}E_{\kappa,\alpha}[\rho(t)]=-\mathcal{L}_{\kappa,\alpha}%
[\rho(t)]\leq0
\]
where
\[
\mathcal{L}_{\kappa,\alpha}[\rho]:=\int_{\mathbb{R}}\left(  \left(  \kappa
\log\rho(x)+\alpha\frac{x^{2}}{2}-\int_{\mathbb{R}}\tilde{W}(x-y)\rho
(y)dy\right)  _{x}\right)  ^{2}\rho(x)\,dx\,.
\]

Let us remark that the functional $\mathcal{W}$ is also known as the
logarithmic energy of $\rho$ as introduced and deeply analyzed in \cite{Saff}
(see also \cite{Voiculescu II, Voiculescu IV}). The next lemma shows the lower
semi-continuity of the functionals $\mathcal{U}(\rho)$, $\mathcal{V}(\rho)$,
and $\mathcal{W}(\rho)$, and as a consequence, of the functional
$E_{\kappa,\alpha}$.

\begin{lem}
\label{weak} The functionals $\mathcal{U}$, $\mathcal{V}$, and $\mathcal{W}$
are lower semi-continuous in $\mathcal{P}_{2}(\mathbb{R})$ with respect to
$d_{2}$. Moreover, the functionals $\mathcal{U}$ and $E_{0,\alpha}$ with
$\alpha>0$ are weak-$*$ lower semi-continuous in $\mathcal{P}_{2}(\mathbb{R})$.
\end{lem}

\textbf{Proof.} The weak-$*$ lower semi-continuity of $\mathcal{U}$ is proven
in \cite[Lemma 3.4]{McCann}, which implies the $d_{2}$ lower semicontinuity.
The weak-$*$ lower semi-continuity of $\mathcal{V}$ is straightforward from
properties of weak-$*$ sequences and it is trivially continuous for the
$d_{2}$ topology.

Before starting the proof for $\mathcal{W}$, let us comment that it is
essentially contained in \cite[Lemma 3.6]{McCann}, although the author deals
with a more regular interaction potential $W(x)$. The proof is inspired from
arguments in \cite[Theorem 1.3]{Saff}. Here, we included it for completeness.
Let us consider the functional
\[
E_{0,\alpha}[\rho]:=\alpha\mathcal{V}(\rho)+\mathcal{W}(\rho)=\int
_{\mathbb{R}^{2}}R(x,y)\,\rho(x)\rho(y)\,dx\,dy\,,
\]
with $\alpha>0$ and
\[
R(x,y):=%
\begin{cases}
-\dfrac{1}{2\pi}\log\left(  |x-y|e^{-\frac{\alpha\pi(x^{2}+y^{2})}{2}}\right)
& \mbox{for }x\neq y\\[2mm]%
+\infty & \mbox{for }x=y
\end{cases}
\,.
\]
Since the function $R(x,y)\rightarrow\infty$ as $|(x,y)|\rightarrow\infty$ and
as $|x-y|\rightarrow0$, it is obviously smooth except at the diagonal and
bounded from below, then it can be approximated pointwise by an increasing
sequence of functions $R_{k}(x,y)\in C_{0}^{\infty}(\mathbb{R}\times
\mathbb{R})$ as $k\rightarrow\infty$. If $\rho_{n}\rightarrow\rho$ weak-$\ast$
as measures, then certainly the product measure $\rho_{n}\times\rho_{n}$
converges to $\rho\times\rho$ weak-$\ast$ as measures in $\mathbb{R}%
\times\mathbb{R}$. Now, define
\[
E_{0,\alpha}^{k}[\rho]:=\int_{\mathbb{R}^{2}}R_{k}(x,y)\,\rho(x)\rho
(y)\,dx\,dy\,,
\]
and note that $E_{0,\alpha}^{k}[\rho_{n}]\leq E_{0,\alpha}[\rho_{n}]$, for all
$n\in\mathbb{N}$. Then, due to the weak-$\ast$ convergence, we get
\[
E_{0,\alpha}^{k}[\rho]=\lim_{n\rightarrow\infty}E_{0,\alpha}^{k}[\rho_{n}%
]\leq\liminf_{n\rightarrow\infty}E_{0,\alpha}[\rho_{n}],
\]
for fixed $k\in\mathbb{N}$. On the other hand, by monotone convergence
$E_{0,\alpha}^{k}[\rho]\rightarrow E_{0,\alpha}[\rho]$ as $k\rightarrow\infty$
and we obtain $E_{0,\alpha}[\rho]\leq\liminf_{n\rightarrow\infty}E_{0,\alpha
}[\rho_{n}]$. The remaining statements follow from the continuity of
$\mathcal{V}$ in the $d_{2}$ topology.\hfill$\diamond$\vskip12pt

\begin{remark}
\label{domain} Let us note that the domain of the functional $\mathcal{W}$
consists only of absolutely continuous measures with respect to Lebesgue
$D(\mathcal{W})\subset\mathcal{P}_{2}^{ac}(\mathbb{R})$. This is a consequence
of the definition of $\tilde{W}$ and the fact that given a measure $\mu$ with
atomic or singular part in its Lebesgue decomposition, then $\mu\times\mu$
will charge the diagonal with positive measure. Notice that $\mathcal{P}%
_{2}\cap L^{1}\cap L^{\infty}(\mathbb{R})\subset D(\mathcal{W}).$
\end{remark}

\subsection{Minimizing the inviscid free energy functional}

The aim of this section is to make a summary about how to show the existence
of a unique minimum among all probability measures in $\mathcal{P}%
_{2}(\mathbb{R})$ to the free energy functional
$P[\rho]:=E_{0,1}[\rho]$. By Lemma \ref{weak}, we already know
that $P$ is weak-* lower semi-continuous, and then, in order to
ensure the existence of a minimum, we only need to show that the
functional is bounded from below. Uniqueness, compact support,
characterization and the explicit form of the minimum of this
functional were studied in relation to the logarithmic capacity of
sets, free probability and connections to random matrices in
\cite{Saff,HP,Voiculescu II, Voiculescu IV}. We will prove some of
them for the sake of the reader in the next proposition.

\begin{proposition}
\label{min} \textrm{\cite{Saff,HP}} Let $\vartheta:=\inf\left\{  P[\rho];
\rho\in\mathcal{P}_{2}(\mathbb{R})\right\}  $. Then:

\begin{enumerate}
\item[i)] $\vartheta$ is finite.

\item[ii)] There is a unique $\bar\rho\in\mathcal{P}_{2}^{ac}(\mathbb{R})$
such that $P[\bar\rho]=\vartheta$ with compact support.

\item[iii)] Moreover, one can characterize $\bar\rho$ as the unique measure in
$\mathcal{P}_{2}^{ac}(\mathbb{R})$ satisfying that
\[
\frac{x^{2}}{2}-\int_{\mathbb{R}}\log|x-y| \bar\rho(y)\,dy \geq C_{\bar\rho}
\]
a.e. $x\in\mathbb{R}$ with equality on supp$(\bar\rho)$ and with
\[
C_{\bar\rho}:=2\vartheta-\int_{\mathbb{R}}\frac{x^{2}}{2} \bar\rho(x)\,dx\, .
\]

\item[iv)] Furthermore, the minimum can be explicitly computed by using the
previous characterization and is given by the semicircle law, i.e., $\bar\rho$
is the absolutely continuous measure with respect to Lebesgue with density
given by
\[
\bar\rho(x)\, dx = \frac1{\pi} \sqrt{(2-x^{2})_{+}} \, dx\, .
\]

\end{enumerate}
\end{proposition}

\textbf{Proof.} \textit{Part i):} We first show that $P[\rho]>0$ for all
$\rho\in\mathcal{P}_{2}^{ac}(\mathbb{R})$ implying that $\vartheta>-\infty$.
For this, observe that, for all $(x,y)\in\mathbb{R}^{2}$,
\[
0\leq|x-y|e^{-\frac{(x^{2}+y^{2})}{2}}\leq(|x|+|y|)e^{-\frac{\left(
|x|+|y|\right)  ^{2}}{4}}\leq\sup_{r\geq0}re^{-\frac{r^{2}}{4}}=\left(
\frac{2}{e}\right)  ^{1/2}<1.
\]
Thus,
\[
-\log\left(  |x-y|e^{-\frac{(x^{2}+y^{2})}{2}}\right)  \geq\frac{1}{2}%
\log(e/2)>0
\]
and
\begin{align}
P[\rho]  &  =\frac{1}{2}\int_{\mathbb{R}^{2}}-\log\left(  |x-y|e^{-\frac
{(x^{2}+y^{2})}{2}}\right)  d\rho(x)d\rho(y)\label{positivity}\\
&  \geq\int_{\mathbb{R}^{2}}\frac{1}{2}\log(e/2)d\rho(x)d\rho(y)=\frac{1}%
{2}\log(e/2)>0\,.\nonumber
\end{align}
Therefore, $\vartheta>0$. Choose $\rho_{12}=dx|_{_{(1,2)}}\in\mathcal{P}%
_{2}^{ac}(\mathbb{R})$, where $dx$ denotes the Lebesgue measure. Observe that
$-\log|x-y|\geq0$ on $(1,2)\times(1,2)$. Noting that $\sup\{\left\vert
2-y\right\vert ,\left\vert 1-y\right\vert \}$ $\leq1$ when $y\in(1,2)$, then
we obtain from Tonelli theorem that
\[
-\int_{\mathbb{R}^{2}}\log|x-y|\,d(\rho_{12}\times\rho_{12})=-\int
_{(1,2)}\left(  \int_{(1,2)}\log|x-y|dx\right)  dy\leq-\int_{(1,2)}\left(
\int_{(0,1)}\log zdz\right)  dy<\infty\,.
\]
We conclude that $\vartheta<+\infty$.

\textit{Part ii):} The proof follows closely \cite[Theorem I.1.3]{Saff}. We
start by showing that given a probability measure, we can always construct
another compactly supported measure that lowers the energy. First, we observe
that for a sequence $(x_{n},y_{n})_{n=1}^{\infty}$ with $\displaystyle\lim
_{n\rightarrow\infty}\left(  |x_{n}|+|y_{n}|\right)  =\infty,$ we have
\[
0\leq|x_{n}-y_{n}|e^{-\frac{(x_{n}^{2}+y_{n}^{2})}{2}}\leq\left(
|x_{n}|+|y_{n}|\right)  e^{-\frac{\left(  |x_{n}|+|y_{n}|\right)  ^{2}}{4}%
}\rightarrow0,\hspace{0.2cm}\mbox{as}\hspace{0.2cm}n\rightarrow\infty,
\]
and then
\[
\lim_{n\rightarrow\infty}\log\left[  |x_{n}-y_{n}|e^{-\frac{(x_{n}^{2}%
+y_{n}^{2})}{2}}\right]  ^{-1}=+\infty.
\]
Therefore, there exists sufficiently small $\varepsilon>0$, such that
\begin{equation}
-\log|x-y|e^{-\frac{(x^{2}+y^{2})}{2}}>2(\vartheta+1)\text{ \ \ if
\ }(x,y)\notin\Sigma_{\varepsilon}\times\Sigma_{\varepsilon}, \label{log1}%
\end{equation}
with $\Sigma_{\varepsilon}:=\{x\in\mathbb{R};$ $e^{-x^{2}/2}\geq\varepsilon\}$.

Next, we claim that if $\rho\in\mathcal{P}(\mathbb{R}),$ with supp$(\rho
)\cap(\mathbb{R}\backslash\Sigma_{\varepsilon})\neq\emptyset$ and
$P(\rho)<\vartheta+1,$ then there exists a $\tilde{\rho}\in\mathcal{P}%
(\Sigma_{\varepsilon})$ such that $P(\tilde{\rho})<P(\rho).$ Note this implies
that there exists $\varepsilon>0,$ such that
\begin{equation}
\vartheta=\inf\left\{  P(\rho);\rho\in\mathcal{P}(\Sigma_{\varepsilon
})\right\}  . \label{infff}%
\end{equation}
Thus $P(\rho)=\vartheta$ is possible only for measures $\rho$ with support in
$\Sigma_{\varepsilon}$.

Now, observe that (\ref{log1}) and \eqref{positivity} together with
$P(\rho)<\vartheta+1$ implies $\rho(\Sigma_{\varepsilon})>0.$ This allows us
to define
\[
\tilde{\rho}=\dfrac{\rho_{|_{\Sigma_{\varepsilon}}}}{\rho(\Sigma_{\varepsilon
})}\,.
\]
Moreover, we have
\begin{align*}
P[\rho] =  &  \,\frac{1}{2}\left(  \int_{\Sigma_{\varepsilon}\times
\Sigma_{\varepsilon}}-\log\left[  |x-y|e^{-\frac{x^{2}+y^{2}}{2}}\right]
\rho(x)\rho(y)dxdy\right) \\
&  +\frac{1}{2}\left(  \int_{(\Sigma_{\varepsilon}\times\Sigma_{\varepsilon
})^{c}}-\log\left[  |x-y|e^{-\frac{x^{2}+y^{2}}{2}}\right]  \rho
(x)\rho(y)dxdy\right) \\
>  &  \,\rho(\Sigma_{\varepsilon})^{2}\left(  \frac{1}{2}\int_{\Sigma
_{\varepsilon}\times\Sigma_{\varepsilon}}-\log\left[  |x-y|e^{-\frac
{x^{2}+y^{2}}{2}}\right]  \tilde{\rho}(x)\tilde{\rho}(y)dxdy\right) \\
&  +\frac{1}{2}\int_{\mathbb{R}^{2}\backslash{\Sigma_{\varepsilon}\times
\Sigma_{\varepsilon}}}2(\vartheta+1)\rho(x)\rho(y)dxdy\\
=  &  \,P[\tilde{\rho}]\rho(\Sigma_{\varepsilon})^{2}+(\vartheta
+1)(1-\rho(\Sigma_{\varepsilon})^{2}).\\
>  &  \,P[\tilde{\rho}]\rho(\Sigma_{\varepsilon})^{2}+P[\rho](1-\rho
(\Sigma_{\varepsilon})^{2}),
\end{align*}
since $P(\rho)<\vartheta+1$. Hence, $P[\rho]>P[\tilde{\rho}]$ and the claim follows.

As a consequence of \eqref{infff}, if $\rho$ is a minimum for $P$, then $\rho$
has compact support in $\Sigma_{\varepsilon}$, and thus $\rho\in
\mathcal{P}_{2}(\mathbb{R})$.

A standard argument in calculus of variations now shows that the minimum is
attained in the set $\mathcal{P}_{2}(\mathbb{R})$. By definition of
$\vartheta$ and \eqref{infff}, there exists a minimizing sequence, i.e.,
$\{\rho_{n}\}\subseteq\mathcal{P}_{2}(\mathbb{R})$ with $P[\rho_{n}%
]\rightarrow\vartheta$ as $n\rightarrow\infty$ with supp$(\rho_{n}%
)\subset\Sigma_{\varepsilon}$ for all $n\in\mathbb{N}$. Note that each
$\rho_{n}$ has support in the compact $\Sigma_{\varepsilon}$, and then we have
that the minimizing sequence of measures is tight in the weak convergence of
measures. Therefore, we can select from $\{\rho_{n}\}_{n\in\mathbb{N}}$ a
weak$^{\ast}$ convergent subsequence and without loss of generality, we can
assume that $\{\rho_{n}\}_{n\in\mathbb{N}}$ itself converges to $\rho
\in\mathcal{P}_{2}(\mathbb{R})$ in the weak$^{\ast}$ topology of measures and
in the $d_{2}$ sense. Therefore, from weak$-\ast$ semi-continuity of $P$, we
have
\[
\vartheta\leq P[\rho]\leq\lim\inf_{n\rightarrow\infty}P[\rho_{n}]=\vartheta
\]
and thus $P[\rho]=\vartheta$. The absolutely continuity of the minimum is a
direct consequence of the Remark \ref{domain} since $D(\mathcal{W}%
)\subset\mathcal{P}_{2}^{ac}(\mathbb{R})$.

The uniqueness of the minimum is proven in \cite[Theorem I.1.3]{Saff}.
However, in our context this will be clear later on from convexity properties,
so we postpone this discussion.

\textit{Parts iii) and iv):} The characterization of the minimum is due to the
Euler-Lagrange equations of the variational problem with the mass constraint.
However, since the minimum has compact support, then one obtains this
variational inequality outside its support. Moreover, this characterization
allows to find explicitly the minimum given by the absolutely continuous
measure with density defining the semicircle law in (iv). We refer for all
details to \cite[Theorem I.1.3 and Theorem IV.5.1]{Saff} and \cite{HP} since
it is a well-known fact in free probability and logarithmic capacity.
\hfill$\diamond$\vskip12pt

\begin{remark}
An important consequence of this result is that there exists a unique
compactly supported stationary solution of problem \eqref{eq4.1} in
$\mathcal{P}_{2}(\mathbb{R})$ explicitly given by the semicircle law
\textrm{\cite{HP}} or the Barenblatt-Pattle profile for $m=3$ of nonlinear
diffusions \textrm{\cite{CT}}. Therefore, using the self-similar change of
variables \eqref{change}, the problem \eqref{eq3} admits a unique, up to
invariance and translations, global in time self-similar solution with initial
data in $\mathcal{P}_{2}(\mathbb{R})$. This is already obtained and studied in
\textrm{\cite{Deslippe,Biler-Karch}}.
\end{remark}

\subsection{The viscous case $\kappa>0$}

In this subsection, we are concerned with the study of the functional
$E_{\kappa,\alpha}$ for $\kappa>0$. Our intent is to show that functional
reaches a unique minimum point on $\mathcal{P}_{2}^{ac}(\mathbb{R})$. As we
already discussed before, a suitable notion of convexity of the functional in
the set of measures will be very important in this case. Next, we recall the
definition of convexity along generalized geodesics of a functional
$E:\mathcal{P}_{2}(\mathbb{R})\longrightarrow\mathbb{R}$.

\begin{definition}
\textrm{\cite{Ambrosio}} A generalized geodesic connecting $\rho$ to $\mu$
(with base in $\nu$ and induced by $\gamma)$ is a curve of the type
$g_{t}=\left(  \pi_{t}^{2\rightarrow3}\right)  \sharp\gamma,$ $t\in
\lbrack0,1],$ where $\gamma\in\Gamma(\nu,\rho,\mu)$, $\pi_{\sharp}^{1,2}%
\gamma\in\Gamma_{0}(\nu,\rho)$, $\pi_{\sharp}^{1,3}\gamma\in\Gamma_{0}(\nu
,\mu)$, and $\pi_{t}^{2\rightarrow3}=(1-t)\pi_{2}+t\pi_{3}$.
\end{definition}

In particular, when dealing with absolutely continuous measures $\rho,\mu
,\nu\in\mathcal{P}_{2}^{ac}(\mathbb{R})$ and with $\rho=\nu$, $g_{t}:=\left[
(1-t)Id+tT_{\rho}^{\mu}\right]  _{\sharp}\rho$ is a generalized geodesic
connecting $\rho$ to $\mu.$ In this case, we call $g_{t}$ the displacement
interpolation between $\rho$ and $\mu.$

\begin{definition}
\textrm{\cite{Ambrosio}} A functional $E:\mathcal{P}_{2}(\mathbb{R}%
)\rightarrow(-\infty,+\infty]$ is $\lambda-$convex along generalized geodesics
(a.g.g. by shorten) if for every $\nu,\rho,\mu\in D(E):=\{\mu\in
\mathcal{P}_{2}(\mathbb{R});E[\mu]<\infty\}\subset\mathcal{P}_{2}(\mathbb{R})$
and for every generalized geodesic $g_{t}$ connecting $\rho$ to $\mu$ induced
by a plan $\gamma\in\Gamma(\nu,\rho,\mu)$, the following inequality holds:
\[
E[g_{t}]\leq(1-t)E[\rho]+tE[\mu]-\frac{\lambda}{2}t(1-t)\,d_{\gamma}^{2}%
(\rho,\mu),
\]
where
\[
d_{\gamma}^{2}(\rho,\mu):=\int_{\mathbb{R}^{3}}\left\vert x_{3}-x_{2}%
\right\vert ^{2}d\gamma(x_{1},x_{2},x_{3})\geq d_{2}^{2}(\rho,\mu)\,.
\]
If $g_{t}$ is the displacement interpolation and $\lambda=0$, we say that the
functional $E$ is displacement convex as originally introduced in
\textrm{\cite{McCann}}.
\end{definition}

We readily apply these notions of convexity to our functional $E_{\kappa
,\alpha}$.

\begin{proposition}
Let $\kappa,\alpha\geq0$. The functional $E_{\kappa,\alpha}$ defined by
\eqref{func3} is $\alpha-$convex along generalized geodesics.
\end{proposition}

\textbf{Proof.} Following the notation in \eqref{func3}, we can reduce
ourselves to show that the functional $\mathcal{W}$ is convex along
generalized geodesics in $D(E_{\kappa,\alpha})\subset\mathcal{P}_{2}%
^{ac}(\mathbb{R})$. In fact, it is well-known that $\mathcal{U}$ is convex
($\lambda=0)$ along generalized geodesics, and that the functional
$\mathcal{V}$ is $1-$ convex along generalized geodesics in $D(E_{\kappa
,\alpha})\subset\mathcal{P}_{2}^{ac}(\mathbb{R})$, see \cite{Ambrosio,McCann}
for details.

Let $\rho,\mu,\nu\in$ $D(E_{\kappa,\alpha})\subset\mathcal{P}_{2}%
^{ac}(\mathbb{R})$ and $g_{t}$ be a generalized geodesic connecting $\rho$ to
$\mu$ with base point $\nu.$ As we are working on the real line $\mathbb{R},$
we can express the generalized geodesics as
\[
g_{t}=((1-t)T_{\nu}^{\rho}+tT_{\nu}^{\mu})_{\#}\nu,
\]
where $T_{\nu}^{\rho}$ and $T_{\nu}^{\mu}$ are the optimal transport between
$\nu$ and $\rho$, and $\nu$ and $\mu$ respectively, with the properties in
Lemma \ref{tnondec}.

Let us observe that by definition of $g_{t}$ and its absolute continuity with
respect to Lebesgue, we get
\begin{align*}
\mathcal{W}[g_{t}]  &  =\int_{\mathbb{R}^{2}}-\log(|x-y|)\,d(g_{t}\times
g_{t})\\
&  =\int_{\mathbb{R}^{2}}-\log(|(1-t)(T_{\nu}^{\rho}(x)-T_{\nu}^{\rho
}(y))+t(T_{\nu}^{\mu}(x)-T_{\nu}^{\mu}(y))|)\,d(\nu\times\nu)\\
&  \leq(1-t)\int_{\mathbb{R}^{2}}-\log(|(T_{\nu}^{\rho}(x)-T_{\nu}^{\rho
}(y))|)\,d(\nu\times\nu)+t\int_{\mathbb{R}^{2}}-\log(|(T_{\nu}^{\mu}%
(x)-T_{\nu}^{\mu}(y))|)\,d(\nu\times\nu)\\
&  =(1-t)\mathcal{W}[\rho]+t\mathcal{W}[\mu],
\end{align*}
where the last step follows from convexity of the function $-\log|x|$. This
last step is fully rigorous provided the following claim (C) holds: there
exists a $\nu$-null set $A$ such that
\[
(1-t)(T_{\nu}^{\rho}(x)-T_{\nu}^{\rho}(y))+t(T_{\nu}^{\mu}(x)-T_{\nu}^{\mu
}(y))\neq0,
\]
for all $t\in\lbrack0,1],$ $x,y\in A^{c}$ and $x\neq y$. In other words, the
interpolation map reaches the logarithmic singularity only if $x=y$ or in a
$\nu\times\nu$-null set.

In order to prove this claim, we remind that the optimal transport on the real
line between two measures in $\mathcal{P}_{2}^{ac}(\mathbb{R})$ is essentially
increasing, see Lemma \ref{tnondec}. Now, let $A$ be a $\nu$-null set such
that $T_{\nu}^{\mu}$ and $T_{\nu}^{\rho}$ are increasing in $A^{c}.$ If
$\ x,y\in A^{c}$ and $x\neq y$ then let us show that
\[
(1-t)(T_{\nu}^{\rho}(x)-T_{\nu}^{\rho}(y))+t(T_{\nu}^{\mu}(x)-T_{\nu}^{\mu
}(y))\neq0,\text{ }\forall t\in\lbrack0,1].
\]
To prove this, suppose that $\exists$ $t_{0}\in(0,1]$, $x,y\in A^{c}$ and
$x\neq y$ such that
\begin{equation}
(1-t_{0})(T_{\nu}^{\rho}(x)-T_{\nu}^{\rho}(y))+t_{0}(T_{\nu}^{\mu}(x)-T_{\nu
}^{\mu}(y))=0\,, \label{aux1212}
\end{equation}
then we deduce
\[
\frac{(T_{\nu}^{\mu}(x)-T_{\nu}^{\mu}(y))}{(T_{\nu}^{\rho}(x)-T_{\nu}^{\rho
}(y))}=\frac{t_{0}-1}{t_{0}}\leq0,
\]
which provides a contradiction, because the optimal transport maps $T_{\nu
}^{\mu}$ and $T_{\nu}^{\rho}$ are increasing in $A^{c}$. For the case
$t_{0}=0$ in (\ref{aux1212}), we have $T_{\nu}^{\rho}(x)-T_{\nu}^{\rho}(y)=0$
which yields $x=y$ or $x,y\in A,$ because of essentially injectivity of
$T_{\nu}^{\rho}$. This finally shows the claim (C), and thus that
$\mathcal{W}[g_{t}]$ is a convex function in $t\in\lbrack0,1]$ for all
generalized geodesics corresponding to absolutely continuous measures, which
gives by definition the convexity of $\mathcal{W}$ along generalized geodesics
in $D(E_{\kappa,\alpha})$. \hfill$\diamond$\vskip12pt

\begin{proposition}
\label{min2} Let $\kappa\geq0$, $\alpha>0$, and $\vartheta_{\kappa,\alpha
}:=\inf\left\{  E_{\kappa,\alpha}(\rho);\rho\in\mathcal{P}_{2}(\mathbb{R}%
)\right\}  $. Then:

\begin{enumerate}
\item[i)] $\vartheta_{\kappa,\alpha}$ is finite.

\item[ii)] There is a unique $\bar\rho_{\kappa,\alpha}\in\mathcal{P}_{2}^{ac}
(\mathbb{R})$ such that $E_{\kappa,\alpha}[\bar\rho_{\kappa,\alpha}%
]=\vartheta_{\kappa,\alpha}$.
\end{enumerate}
\end{proposition}

\textbf{Proof. } This results is proven for $\kappa=0$ in Proposition
\ref{min}, except the uniqueness part. Let us assume from now that $\kappa>0$.

\emph{Part i):} Recalling the definition of the functional $E_{\kappa,\alpha}$
in terms of $\mathcal{U}$, $\mathcal{V}$, and $\mathcal{W}$ we split
\[
E_{\kappa,\alpha}=E_{0,\alpha/2}+\kappa\mathcal{U}+\frac{\alpha}{2}%
\mathcal{V}\,.
\]
It is straightforward to use Jensen's inequality to show that
\[
\kappa\mathcal{U}[\rho]+\frac{\alpha}{2}\mathcal{V}[\rho]=\kappa
\int_{\mathbb{R}}\frac{\rho}{e^{-\delta x^{2}/2}}\log\left(  \frac{\rho
}{e^{-\delta x^{2}/2}}\right)  e^{-\delta x^{2}/2}dx\geq\frac{\kappa}{2}%
(\log\delta-\log(2\pi)),\text{ }\forall\rho\in\mathcal{P}_{2}^{ac}%
(\mathbb{R})
\]
with $\delta=\alpha/2\kappa$. Proceeding analogously to the proof of
\textit{Part i)} of Proposition \ref{min}, we obtain
\[
0\leq|x-y|e^{-\alpha\frac{(x^{2}+y^{2})}{4}}\leq(|x|+|y|)e^{-\alpha
\frac{\left(  |x|+|y|\right)  ^{2}}{8}}\leq\sup_{r>0}re^{-\alpha\frac{r^{2}%
}{8}}=\left(  \frac{4}{\alpha e}\right)  ^{1/2}\text{ ,}%
\]
and
\begin{equation}
E_{0,\alpha/2}[\rho]=\int_{\mathbb{R}^{2}}-\log\left(  |x-y|e^{-\alpha
\frac{(x^{2}+y^{2})}{4}}\right)  d\rho(x)d\rho(y)\geq-\frac{1}{2}\log\left(
\frac{4}{\alpha e}\right)  . \label{unifBaixo}%
\end{equation}
Therefore, the functional $E_{\kappa,\alpha}$ is bounded from below, and so
$\vartheta_{\kappa,\alpha}>-\infty$. Since the domain $D(E_{\kappa,\alpha
})\neq\varnothing,$ there is $\rho\in\mathcal{P}_{2}(\mathbb{R})$ such that
$E_{\kappa,\alpha}(\rho)<\infty,$ and therefore the infimum $\vartheta
_{\kappa,\alpha}$ is finite.

\emph{Part ii):} From Lemma \ref{weak}, $E_{\kappa,\alpha}$ is weak-$\ast$
semi-continuous in $\mathcal{P}_{2}^{ac}(\mathbb{R})$, which shows that the
infimum is achieved at some point $\bar{\rho}\in\mathcal{P}_{2}^{ac}%
(\mathbb{R})$. In fact, it is easy to check based on the same arguments for
\textit{Part i) and }Proposition \ref{min} that any minimizing sequence is
weakly compact in $L^{1}(\mathbb{R})$, see similar arguments in
\cite{Carrillo2}, since $\kappa>0$.

The uniqueness claim for $\kappa\geq0$ follows from the strict displacement
convexity of $E_{\kappa,\alpha}$. Indeed, let $\rho_{1}$ and $\rho_{2}$ be two
different minimum in $\mathcal{P}_{2}^{ac}(\mathbb{R})$ to $E_{\kappa,\alpha}$
and consider $\rho_{\frac{1}{2}}$ the displacement interpolation between
$\rho_{1}$ and $\rho_{2}$ at $t=1/2$. By the $\alpha$-displacement convexity
of $E_{\kappa,\alpha}$, we have
\[
\vartheta_{\kappa,\alpha}\leq E_{\kappa,\alpha}[\rho_{\frac{1}{2}}]<\frac
{1}{2}E_{\kappa,\alpha}[\rho_{1}]+\frac{1}{2}E_{\kappa,\alpha}[\rho
_{2}]=\vartheta_{\kappa,\alpha},
\]
which provides a contradiction. Therefore, there exists a unique minimum of
$E_{\kappa,\alpha}$. \hfill$\diamond$\vskip12pt


\section{Well-posedness, asymptotic behavior and inviscid limit}

As pointed out in the introduction we will obtain solutions for (\ref{eq4.1})
as the limit of a Euler approximation scheme in probability space
$\mathcal{P}_{2}(\mathbb{R})$. More precisely, consider a step time $\tau>0$
and a initial data $\rho_{0}\in\mathcal{P}_{2}(\mathbb{R})$. We define, for a
fixed $\mu,$ the functional $\mathcal{I}(\tau,\mu,.):\mathcal{P}%
_{2}(\mathbb{R})\rightarrow(-\infty,\infty]$ as
\[
\mathcal{I}(\tau,\mu;\rho):=\frac{1}{2\tau}d_{2}^{2}(\mu,\rho)+E[\rho].
\]
Formally, we define the following recursive sequence $(\rho_{\tau}^{n}%
)_{n\in\mathbb{N}}$:
\begin{align}
\rho_{\tau}^{0}  &  :=\rho_{0}\label{v1}\\
\rho_{\tau}^{n}  &  =\min_{\rho\in\mathcal{P}_{2}(\mathbb{R})}I(\tau
,\rho_{\tau}^{n-1};\rho),\text{ \ \ }n\in\mathbb{N}, \label{v2}%
\end{align}
which can be seen as the discrete approximate Euler solution to the gradient
flux equation
\[
\frac{\partial\rho}{\partial t}=-\nabla E[\rho]\text{, \ \ \ }t>0
\]
in the metric space $(\mathcal{P}_{2}(\mathbb{R}),$ $d_{2}).$ More precisely,
one calls a \textit{discrete solution}, the curve $\rho_{t}^{\tau}$ obtained
as the time interpolation of the discrete scheme (\ref{v1})-(\ref{v2})
connecting every pair $(\rho_{\tau}^{n-1},\rho_{\tau}^{n})$ with a velocity
constant geodesic in $t\in\lbrack(n-1)\tau,n\tau)$, see \cite{Ambrosio}.

\subsection{Gradient flows}

Below we remember the definition of a gradient flow solution.

\begin{definition}
\label{gradflow} We say that a map $\rho_{t}\in AC_{loc}^{2}((0,\infty
);\mathcal{P}_{2}(\mathbb{R})$\thinspace$)$ is a solution of the gradient flow
equation
\[
-v_{t}\in\partial E(\rho_{t}),\text{ \ \ }t>0,
\]
if \ $-v_{t}\in Tan_{\rho_{t}}\mathcal{P}_{2}(\mathbb{R})$ belongs to the
subdifferential of $E$ at $\rho_{t},$ a.e. $t>0.$
\end{definition}

\bigskip

It is known that $\rho_{t}$ being a gradient flow in $\mathcal{P}%
_{2}(\mathbb{R})$ is equivalent to the existence of a velocity vector field
$-v_{t}\in Tan_{\rho_{t}}\mathcal{P}_{2}(\mathbb{R})\cap$ $\partial E(\rho
_{t})$ a.e. $t>0$, such that $\left\Vert v_{t}\right\Vert _{L_{\rho_{t}}%
^{2}(\mathbb{R})}\in L_{loc}^{2}(0,\infty)$ and the continuity equation holds
in the distribution sense:
\begin{equation}
\frac{\partial(\rho_{t})}{\partial t}+\nabla\cdot(\rho_{t}v_{t})=0\text{
\ \ \ \ in \ \ \ }\mathbb{R}\times\mathbb{(}0,\infty). \label{partdist}%
\end{equation}
The next theorem ensures the existence of a gradient flow solution for the
free energy functional $E_{\kappa,\alpha}$ as in (\ref{func3}).

\begin{thm}
\label{teo1} Let $\kappa,\alpha\geq0,$ $\rho_{0}\in\mathcal{P}_{2}%
(\mathbb{R})$ and the functional $E_{\kappa,\alpha}$. The following assertions hold:

\begin{enumerate}
\item \textbf{(Existence and Uniqueness)} The discrete solution $\rho
_{t}^{\tau}$ converges locally uniformly to a locally Lipschitz curve
$\rho_{t}:=S_{t}[\rho_{0}]$ in $\mathcal{P}_{2}(\mathbb{R})$ which is the
unique gradient flow of $E_{\kappa,\alpha}$ with $\lim_{t\rightarrow0+}%
\rho_{t}=\rho_{0}$. Moreover, the curve lies in $\mathcal{P}_{2}%
^{ac}(\mathbb{R})$, for all $t>0.$

\item \textbf{(Contractive semigroup)} The map $t\mapsto S_{t}[\rho_{0}]$ for
all $\alpha\geq0$ is a $\alpha$-contracting semigroup on $\mathcal{P}%
_{2}(\mathbb{R})$, i.e.
\[
d_{2}(S_{t}[\rho_{0}],S_{t}[\mu_{0}])\leq e^{-\alpha t}\,d_{2}(\rho_{0}%
,\mu_{0})\qquad\text{ for all }\rho_{0},\mu_{0}\in\mathcal{P}_{2}%
(\mathbb{R})\,.
\]

\item \textbf{(Asymptotic behavior) } Let $\alpha>0$ and let us denote by
$\bar\rho_{\kappa,\alpha}$ the unique minimum of $E_{\kappa,\alpha}$. Then for
all $0<t_{0}<t<\infty$, we have
\[
d_{2}(\rho_{t},\overline{\rho}_{\kappa,\alpha}) \leq e^{-\alpha(t-t_{0}%
)}\,d_{2} (\rho_{t_{0}},\overline{\rho}_{\kappa,\alpha})
\]
and
\[
E_{\kappa,\alpha}[\rho_{t}]-E_{\kappa,\alpha}[\overline{\rho}_{\kappa,\alpha}]
\leq e^{-2\alpha(t-t_{0})}(E_{\kappa,\alpha}[\rho_{t_{0}}]-E_{\kappa,\alpha
}[\overline{\rho}_{\kappa,\alpha}]) \, .
\]

\item \textbf{(Free Energy identity)} The solution $\rho_{t}:=S_{t}[\rho_{0}]$
is a curve of maximal slope and it satisfies the identity:
\[
E_{\kappa,\alpha}[\rho_{t}] = E_{\kappa,\alpha}[\rho_{s}] +\int_{s}^{t}
\int_{\mathbb{R}} |v_{\tau}(x)|^{2} dx d\tau
\]
for all $0\leq s\leq t$, where $v_{t} \in L^{2}_{loc}(0,\infty;L^{2}_{\rho
_{t}}(\mathbb{R}))$ is the associated velocity field satisfying
\eqref{partdist} in Definition \textrm{\ref{gradflow}}.
\end{enumerate}
\end{thm}

\bigskip

\textbf{Proof. }First notice that $\mathcal{P}_{2}\cap L^{1}\cap L^{\infty
}\subset D(E_{\kappa,\alpha})\subset\mathcal{P}_{2}^{ac}(\mathbb{R})$ by
Remark \ref{domain} and $\overline{D(E_{\kappa,\alpha})}=\mathcal{P}%
_{2}(\mathbb{R})$. Let us start with the case $\alpha>0$. Collecting the
results obtained through previous sections, we have that the functional
$E_{\kappa,\alpha}:\mathcal{P}_{2}(\mathbb{R})\rightarrow(-\infty,+\infty]$ is
a proper, l.s.c., coercive functional and $\alpha-$convex along generalized
geodesics. Moreover, $\mathcal{I}(\mu,\tau;\rho)$ admits at least a minimum
point $\mu_{\tau},$ for all $\tau\in(0,\tau_{\ast})$ and $\mu\in
P_{2}(\mathbb{R})$. The minimum $\mu_{\tau}\in P_{2}^{ac}(\mathbb{R})$ for
$\kappa>0,$ because minimizing sequences are weakly compact in $L^{1}%
(\mathbb{R})$ as in Proposition \ref{min2} and by the Remark
\ref{domain} if $\kappa=0$. Therefore, all the statements result
directly from the general theory of gradient flows developed in
\cite[Theorem 11.2.1]{Ambrosio}. In case $\alpha=0$, we deal with
plain convex functionals along generalized geodesics and the same
results apply. However, we need to be careful with the coercivity
and the existence of minimizers for the one-step variational
scheme since we lack of a direct confinement. This is easily
provided by the following observation using the triangular
inequality and 
\begin{equation}
\mathcal{I}(\tau,\mu;\rho):=\frac{1}{2\tau}d_{2}^{2}(\mu,\rho)+E_{\kappa
,0}[\rho]\geq\frac{1}{4\tau}d_{2}^{2}(\rho,\delta_{0})-\frac{1}{2\tau}%
d_{2}^{2}(\delta_{0},\mu)+E_{\kappa,0}[\rho]=-\frac{1}{\tau}\mathcal{V}%
[\mu]+E_{\kappa,2/\tau}[\rho] \label{kkk}%
\end{equation}
for all $\mu,\rho\in\mathcal{P}_{2}(\mathbb{R})$ and all $\tau>0$. Therefore,
this implies the boundedness from below, the existence of minimizers of the
one-step variational scheme, and the coercivity in the case of $\alpha=0$.
Again the results of \cite[Theorem 11.2.1]{Ambrosio} apply directly.
\hfill$\diamond$\vskip12pt

\begin{remark}
As a consequence of previous theorem, we have shown the global-in-time
well-posedness for the Cauchy problem for general measures in $\mathcal{P}%
_{2}(\mathbb{R})$ as initial data for equations \eqref{eqprinc}
and \eqref{eq3} and their self-similar counterparts \eqref{eq4.1}.
Moreover, we have shown the convergence towards self-similarity in
the sense expressed in the third part of Theorem
\textrm{\rm\ref{teo1}}. Note that the gradient flows obtained in
Theorem \textrm{\rm\ref{teo1}} for the functionals $E_{\kappa,0}$
and $E_{\kappa,1}$ are equivalent through the change of variables
\eqref{change}. Finally, note that the evolution is defined in a
unique way for any initial data in $\mathcal{P}_{2}(\mathbb{R})$.
However, the evolution flow regularizes instantaneously since it
belongs to $\mathcal{P}_{2}^{ac}(\mathbb{R})$ for all $t>0$. This
is the precise mathematical statement showing that the repulsive
logarithmic interaction potential in one dimension is
\textquotedblleft very repulsive\textquotedblright.
\end{remark}

\begin{remark}
\label{Power}(Power-law potentials) Consider the power-law interaction
potential $W(x)=\left\vert x\right\vert ^{-\beta},$ for $0<\beta<1,$ and its
natural extension
\begin{equation}
\tilde{W}_{\beta}(x)=%
\begin{cases}
\left\vert x\right\vert ^{-\beta} & \mbox{ for }x\neq0\\
+\infty & \mbox{ at }x=0
\end{cases}
.\,\label{power1}%
\end{equation}
Let the free energy functional
$E_{\beta,\kappa,\alpha}:\mathcal{P}_{2} (\mathbb{R})\rightarrow$
$\mathbb{R\cup\{+\infty\}}$ be defined as in \eqref{func3} with
$\tilde{W}_{\beta}$ instead of $\tilde{W}.$ Noting that
$\left\vert x\right\vert ^{-\beta}$ is locally integral when
$0<\beta<1,$ similar arguments as in Propositions {\rm\ref{min}}
and {\rm\ref{min2}} give that the infimum
$\vartheta_{\beta,\kappa,\alpha}$ is finite and unique. Moreover,
$\mathcal{P}_{2}\cap L^{1}\cap L^{\infty}(\mathbb{R})\subset D(E_{\beta,\kappa,\alpha}%
)\subset\mathcal{P}_{2}^{ac}(\mathbb{R})$ and $\overline{D(E_{\beta,\kappa,\alpha}%
)}=\mathcal{P}_{2}(\mathbb{R})$. Since the function $\left\vert
x\right\vert ^{-\beta}$ is convex, we have that
$E_{\beta,\kappa,\alpha}$ is $\alpha-$convex along generalized
geodesics. Thus, the results on gradient flows given by Theorem
{\rm \ref{teo1}} also hold true for \eqref{power1} and we can
take initial measure $\rho_{0}\in\mathcal{P}_{2}(\mathbb{R}).$

For range $\beta\geq1,$ we have $D(\mathcal{W})=\{0\}$ and
$\overline {D(E_{\beta,\kappa,\alpha})}=\{0\}$ and then the theory
trivializes. In particular, if $\rho\in D(\mathcal{W})$ were a
continuous function and $K=\{x\in\mathbb{R}:\rho
(x)>\frac{1}{n}\}$ is a positive measure set, for some
$n\in\mathbb{N}$, then
\begin{equation*}
\infty=\frac{1}{n^{2}}\int_{K\times K}\frac{1}{\left\vert
x-y\right\vert ^{\beta}}dxdy\leq\mathcal{W}(\rho)<\infty\text{,}
\end{equation*}
which gives a contradiction. Therefore, the only continuous
function in the domain is $\rho=0$.
\end{remark}

\begin{thm}
\label{teo2} \textbf{(Inviscid limit)} Let us consider the functionals
$E_{\kappa,\alpha}$ and $E_{0,\alpha}$ with $\alpha\geq0$, corresponding to
viscosity $\kappa>0$ and $\kappa=0$ respectively, and assume that $\rho_{0}\in
D(E_{\epsilon_{0},\alpha})$ with $\epsilon_{0}>0$. If $\rho_{\kappa}(t)$,
$\rho(t)$ are the corresponding gradient flow solutions in $\mathcal{P}%
_{2}(\mathbb{R})$ with initial data $\rho_{0}$, then
\[
\rho_{\kappa}(t)\rightarrow\rho(t)\text{ in }\mathcal{P}_{2}(\mathbb{R})
\]
locally uniformly in $[0,\infty)$, as $\kappa\rightarrow0^{+}.$
\end{thm}

\textbf{Proof.} In view of the stability property of \cite[Theorem
12.2.1]{Ambrosio}, we need only to verify (in a neighborhood of $\kappa=0$)
the equicoercivity of the family of functionals $\{E_{\kappa,\alpha}%
\}_{\kappa\geq0}$ and the uniform boundedness at $\rho_{0}$. More precisely,
we need to show
\[
\sup_{\kappa\in(0,\epsilon_{0})}\ E_{\kappa,\alpha}[\rho_{0}]<\infty
\ \ \ \ \text{and }\inf_{\kappa\in(0,\epsilon_{0}),\text{ }\rho\in
\mathcal{P}_{2}(\mathbb{R})}\frac{1}{2\tau}d_{2}^{2}(\mu,\rho)+E_{\kappa
,\alpha}[\rho]>-\infty,
\]
for some $\tau>0$ and $\mu\in$ $\mathcal{P}_{2}(\mathbb{R})$ and $\epsilon
_{0}>0$. Firstly, observe that $\rho_{0}\in D(E_{\epsilon_{0},\alpha})$
implies $\mathcal{U}[\rho_{0}]<\infty,$ $\mathcal{V}[\rho_{0}]<\infty$ and
$\mathcal{W}[\rho_{0}]<\infty.$ It follows from (\ref{func3}) that $\rho
_{0}\in D(E_{\kappa,\alpha})$ for all $\kappa\in(0,\infty).$ Also,
\[
\sup_{\kappa\in(0,\epsilon_{0})}\ E_{\kappa,\alpha}[\rho_{0}]\leq
\max\{0,\epsilon_{0}\,\mathcal{U}[\rho_{0}]\}+\alpha\mathcal{V}[\rho
_{0}]+\mathcal{W}[\rho_{0}]<\infty\,.
\]

In order to conclude the proof, it remains verify the equicoercivity. By using
\eqref{kkk}, we observe
\[
\frac{1}{2\tau}d_{2}^{2}(\mu,\rho)+E_{\kappa,\alpha}[\rho]\geq\frac{1}{2\tau
}d_{2}^{2}(\mu,\rho)+E_{\kappa,0}[\rho]\geq-\frac{1}{\tau}\mathcal{V}%
[\mu]+E_{\kappa,1/2\tau}[\rho]
\]
for all $\alpha\geq0$. Let us split the functional $E_{\kappa,1/\tau}$ as
\[
E_{\kappa,1/2\tau}[\rho]=\kappa\mathcal{U}[\rho]+\frac{1}{2\tau}\mathcal{V}%
[\rho]+\mathcal{W}[\rho]=E_{0,1/4\tau}[\rho]+\kappa\mathcal{U}[\rho]+\frac
{1}{4\tau}\mathcal{V}[\rho]\,.
\]
Let us now remark that $\kappa\mathcal{U}+\frac{\alpha}{2}\mathcal{V}$ is
bounded from below. Note that it is the relative logarithmic entropy
functional leading to the classical linear Fokker-Planck equation whose
minimum is a Gaussian $M(x)$ determined by
\[
M(x)=\left(  4\pi\frac{\kappa}{\alpha}\right)  ^{-1/2}\exp\left(
-\frac{\alpha x^{2}}{4\kappa}\right)  \,.
\]
Therefore, we get
\begin{equation}
\kappa\mathcal{U}[\rho]+\frac{\alpha}{2}\mathcal{V}[\rho]\geq\kappa
\mathcal{U}[M]+\frac{\alpha}{2}\mathcal{V}[M]=-\frac{1}{2}\log\left(
4\pi\frac{\kappa}{\alpha}\right)  \geq-\frac{1}{2}\log\left(  4\pi
\frac{\epsilon_{0}}{\alpha}\right)  , \label{aux-alpha}%
\end{equation}
for $\kappa\in(0,\epsilon_{0})$. Due to Proposition \ref{min2} and
\eqref{unifBaixo}, $E_{0,1/4\tau}$ is also bounded from below by
$\vartheta_{0,1/4\tau}.$ Using (\ref{aux-alpha}) with
$\alpha=\frac{1}{2\tau}$, we conclude
\[
\inf_{\kappa\in(0,\epsilon_{0}),\text{ }\rho\in\mathcal{P}_{2}(\mathbb{R}%
)}\left(
\frac{1}{2\tau}d_{2}^{2}(\mu,\rho)+E_{\kappa,\alpha}[\rho]\right)
\geq-\frac{1}{\tau}\mathcal{V}[\mu]-\frac{1}{2}\log\left(
8\pi\epsilon _{0}\tau\right)  +\vartheta_{0,1/4\tau}\,.
\]
\hfill$\diamond$\vskip12pt

\subsection{Solutions in sense of distributions.}

An important point is to know whether the gradient flow solutions are
solutions in the sense of distributions. Firstly, let us define the notion of
weak solution which we deal with. We say that a measure $\rho_{t}$ is a weak
solution to equation (\ref{eq4.1}), with initial condition $\rho_{0}$, if for
all $\varphi\in C_{0}^{\infty}(\mathbb{R)}$
\begin{equation}
\frac{d}{dt}\int_{\mathbb{R}}\varphi(x)d\rho_{t}=\kappa\int_{\mathbb{R}%
}\varphi^{\prime\prime}(x)d\rho_{t}-\int_{\mathbb{R}}\varphi^{\prime}%
(x)xd\rho_{t}+\frac{1}{2}\int_{\mathbb{R}^{2}}\frac{\varphi^{\prime
}(x)-\varphi^{\prime}(y)}{x-y}d(\rho_{t}\times\rho_{t}) \label{weak1}%
\end{equation}
in the distributional sense in $(0,\infty)$ with $\rho_{t}\rightharpoonup
\rho_{0}$ weakly-$\ast$ as measures.

In order to obtain a connection between gradient flows and weak solutions, we
need to describe, see \cite{Ambrosio}, the minimal selection of the
subdifferential of $E=E_{\kappa,1}$, that is the set $\partial^{\circ
}E_{\kappa,1}(\rho_{t}).$ For that matter we need to consider the following
functional: for each fixed $\rho\in\mathcal{P}_{2}(\mathbb{R})$ we define
$L_{\rho}:C_{0}^{1}(\mathbb{R)\rightarrow R}$ as%
\begin{equation}
L_{\rho}(\varphi)=\lim_{\delta\rightarrow0^{+}}\int_{|x-y|\geq\delta}\frac
{1}{x-y}\varphi(x)d(\rho\times\rho)=\lim_{\delta\rightarrow0^{+}}\frac{1}%
{2}\int_{|x-y|\geq\delta}\frac{\varphi(x)-\varphi(y)}{x-y}d(\rho\times
\rho)<\infty,\forall\varphi\in C_{0}^{1}(\mathbb{R)}. \label{weak-L}%
\end{equation}
It is straightforward to check that
\begin{equation}
\left\vert L_{\rho}(\varphi)\right\vert \leq\frac{1}{2}\left\Vert
\varphi^{\prime}\right\Vert _{\infty}\int_{\mathbb{R}^{2}}d(\rho\times
\rho)=\frac{1}{2}\left\Vert \varphi^{\prime}\right\Vert _{\infty}\text{ ,}
\label{f1}%
\end{equation}
and therefore $L_{\rho}\in(C_{0}^{1}(\mathbb{R))}^{\ast}.$ So, there exists
$\mu\in\mathcal{M}(\mathbb{R)}$ and a constant $c_{0}$ such that (see
\cite[p.225]{Folland})
\[
L_{\rho}(\varphi)=\int_{\mathbb{R}}\varphi^{\prime}d\mu+c_{0}\varphi
(0)\quad\forall\varphi\in C_{0}^{1}(\mathbb{R)}.
\]
{F}rom (\ref{f1}), we can also see that $c_{0}=0$ and we obtain the following
representation to $L_{\rho}$:
\begin{equation}
L_{\rho}(\varphi)=\int_{\mathbb{R}}\varphi^{\prime}d\mu,\ \forall\varphi\in
C_{0}^{1}(\mathbb{R)}. \label{aux-repres}%
\end{equation}

\begin{lem}
\label{lem1}Let $\kappa\geq0$ and $\mu$ as mentioned above. If a measure
$\rho\in D(E_{\kappa,\alpha})\subset\mathcal{P}_{2}(\mathbb{R})$ belongs to
$D(|\partial E_{\kappa,\alpha}|)$ then we have $(\kappa\rho+\mu)\in
W_{\mathrm{{loc}}}^{1,1}(\mathbb{R})$ and
\begin{equation}
\rho\omega=\partial_{x}(\kappa\rho+\mu)+\alpha\rho x\hspace{0.5cm}\text{for
some }\omega\in L_{\rho}^{2}(\mathbb{R}). \label{eq13}%
\end{equation}
In this case the vector $\omega$ defined by \eqref{eq13} is the minimal
selection in $\partial E_{\kappa,\alpha}[\rho],$ i.e. $\omega=\partial^{\circ
}E_{\kappa,\alpha}[\rho].$
\end{lem}

\bigskip\textbf{Proof.} For each compactly supported smooth test function
$\varphi\in C_{0}^{\infty}(\mathbb{R)}$, let us consider the map
$\psi_{\varepsilon}:=Id+\varepsilon\varphi.$ It is easy check that
$\psi_{\varepsilon\#}\rho\in$ $D(E_{\kappa,\alpha}),$ when $\rho\in
D(E_{\kappa,\alpha})$ and for $\varepsilon>0$ small enough. Since $\rho\in
D(|\partial E_{\kappa,\alpha}|)$ and from the definition of metric slope
$|\partial E_{\kappa,\alpha}|[\rho]$, see \cite{Ambrosio}, we have
\begin{align*}
\left\vert A_{1}(\varphi)+A_{2}(\varphi)+A_{3}(\varphi)\right\vert  &
\,:=\left\vert \kappa\lim_{\varepsilon\rightarrow0}\frac{\mathcal{U}%
[\psi_{\varepsilon\#}\rho]-\mathcal{U}[\rho]}{\varepsilon}+\alpha
\lim_{\varepsilon\rightarrow0}\frac{\mathcal{V}[\psi_{\varepsilon\#}%
\rho]-\mathcal{V}[\rho]}{\varepsilon}+\lim_{\varepsilon\rightarrow0}%
\frac{\mathcal{W}[\psi_{\varepsilon\#}\rho]-\mathcal{W}[\rho]}{\varepsilon
}\right\vert \\
&  \leq|\partial E_{\kappa,\alpha}|[\rho]\lim_{\varepsilon\rightarrow0}%
\frac{d_{2}(\psi_{\varepsilon\#}\rho,\rho)}{\varepsilon}<\infty.
\end{align*}
The terms $A_{1}$ and $A_{2}$ can be exactly treated as in \cite[Chapter
11]{Ambrosio} and one obtains
\[
A_{1}(\varphi)=-\kappa\int_{\mathbb{R}}\varphi^{\prime}\,d\rho\text{ \ and
\ \ }A_{2}(\varphi)=\alpha\int_{\mathbb{R}}x\varphi\,d\rho.\text{\ \ }%
\]
Now, we deal with the term $A_{3}$. Notice that the map
\[
Q(\varepsilon,x,y)=\frac{1}{2}\frac{-\log|(x-y+\varepsilon(\varphi
(x)-\varphi(y)))|-(-\log|x-y|)}{\varepsilon}%
\]
is nondecreasing in $\varepsilon>0,$ for fixed $x,y\neq0.$ As $A_{1}(\varphi)$
and $A_{2}(\varphi)$ are finite, then $A_{3}(\varphi)$ is also finite. By the
monotone convergence theorem, we have
\begin{align*}
A_{3}(\varphi) &  =\lim_{\varepsilon\rightarrow0}\int_{\mathbb{R}^{2}%
}Q(\varepsilon,x,y)d(\rho\times\rho)=-\frac{1}{2}\int_{\mathbb{R}^{2}}%
\frac{\varphi(x)-\varphi(y)}{x-y}d(\rho\times\rho)\\
&  =-L_{\rho}(\varphi)=-\int_{\mathbb{R}}\varphi^{\prime}\,d\mu<\infty.
\end{align*}
Notice that the second integral above is not a singular integral because
$\varphi\in C_{0}^{\infty}(\mathbb{R)}.$ Observing that
\[
\lim_{\varepsilon\rightarrow0}\frac{d_{2}(\psi_{\varepsilon\#}\rho,\rho
)}{\varepsilon}\leq\left\Vert \varphi\right\Vert _{L_{\rho}^{2}(\mathbb{R)}},
\]
we get
\begin{align*}
(A_{1}+A_{2}+A_{3})(\varphi) &  =\int_{\mathbb{R}}-\kappa\varphi^{^{\prime}%
}(x)d\rho+\alpha\int_{\mathbb{R}}x\varphi(x)d\rho-L_{\rho}(\varphi)\\
&  \geq-|\partial E_{\kappa,\alpha}|[\rho]\lim_{\varepsilon\rightarrow0}%
\frac{d_{2}(\psi_{\varepsilon\#}\rho,\rho)}{\varepsilon}\geq-|\partial
E_{\kappa,\alpha}|[\rho]\left\Vert \varphi\right\Vert _{L_{\rho}%
^{2}(\mathbb{R)}}.
\end{align*}
Changing $\varphi$ by $-\varphi$, we finally obtain
\[
\left\vert \int_{\mathbb{R}}(-\kappa\varphi^{\prime}+\alpha x\varphi
)d\rho-L_{\rho}(\varphi)\right\vert \leq|\partial E_{\kappa,\alpha}%
|[\rho]\left\Vert \varphi\right\Vert _{L_{\rho}^{2}(\mathbb{R)}}.
\]
So, there exists $\omega\in L_{\rho}^{2}(\mathbb{R)}$ with $\left\Vert
\omega\right\Vert _{L_{\rho}^{2}(\mathbb{R)}}\leq$ $|\partial E_{\kappa
,\alpha}|[\rho]$ such that
\begin{align}
\int_{\mathbb{R}}\omega\varphi\,d\rho &  =(A_{1}+A_{2}+A_{3})(\varphi
)\nonumber\\
&  =-\left(  \kappa\int_{\mathbb{R}}\varphi^{\prime}\,d\rho+L_{\rho}%
(\varphi)\right)  +\alpha\int_{\mathbb{R}}x\varphi\,d\rho\text{, }%
\forall\varphi\in C_{0}^{\infty}(\mathbb{R)}.\label{f2}%
\end{align}
Thus $\omega\in\partial E_{\kappa,\alpha}[\rho]$ and $\omega$ is the minimal
selection in $\partial E_{\kappa,\alpha}[\rho]$, i.e. $\omega=\partial^{\circ
}E_{\kappa,\alpha}[\rho]$. Finally, let us characterize $\omega$. Since
$\rho\in\mathcal{P}_{2}(\mathbb{R})$ implies that $\psi\lbrack\varphi
]=\int_{\mathbb{R}}x\varphi\,d\rho$ is bounded in $L_{\rho}^{2}(\mathbb{R)}$
(with norm at most $\int_{\mathbb{R}}x^{2}d\rho$), we get
\begin{align*}
\left\vert <\partial_{x}(\kappa\rho+\mu),\varphi>\right\vert  &  =\left\vert
\int_{\mathbb{R}}\varphi^{\prime}\,d(\kappa\rho+\mu)\right\vert \leq\left(
|\partial E_{\kappa,\alpha}|(\rho)+\alpha\int_{\mathbb{R}}x^{2}d\rho\right)
\left\Vert \varphi\right\Vert _{L_{\rho}^{2}(\mathbb{R)}}\\
&  \leq\left(  |\partial E_{\kappa,\alpha}|(\rho)+\alpha\int_{\mathbb{R}}%
x^{2}d\rho\right)  \left\Vert \varphi\right\Vert _{\infty}\,.
\end{align*}
Therefore, $\partial_{x}(\kappa\rho+\mu)\in\mathcal{M}(\mathbb{R{)}},$ i.e.
$(\kappa\rho+\mu)\in BV(\mathbb{R)}.$ Integration by parts holds:
\[
\left\vert \int_{\mathbb{R}}\varphi(x)d(\partial_{x}(\kappa\rho+\mu
))\right\vert \leq\left(  |\partial E_{\kappa,\alpha}|(\rho)+\alpha
\int_{\mathbb{R}}x^{2}d\rho\right)  \left\Vert \varphi\right\Vert _{L_{\rho
}^{2}(\mathbb{R)}},
\]
which implies $\partial_{x}(\kappa\rho+\mu)\in$ $L_{\rho}^{2}(\mathbb{R)\cap
}\mathcal{M}(\mathbb{R)}$ and then $(\kappa\rho+\mu)\in W_{loc}^{1,1}%
(\mathbb{R)}.$ Finally, coming back to (\ref{f2}), we obtain the
following expression for $\omega$, the element of minimal norm in
the subdifferential of $E_{\kappa,\alpha}$:
$\rho\omega=\partial_{x}(\kappa\rho+\mu)+\alpha x\rho$.
\hfill$\diamond$\vskip12pt

The next theorem gives a connection between gradient flows and the notion of
weak solution (\ref{weak1}).

\begin{thm}\label{dist}
\textbf{(Distributional solution) }Let $\mu_{t}$ correspond to
$\rho_{t}$ through (\ref{weak-L}). For every
$\rho_{0}\in\mathcal{P}_{2}(\mathbb{R})$ and
every $\kappa,\alpha\geq0$, the gradient flow $\rho_{t}$ in $\mathcal{P}%
_{2}(\mathbb{R})$ of the functional $E_{\kappa,\alpha}$ is a distributional
solution of the equation
\begin{equation}
\frac{\partial\rho_{t}}{\partial t}=\frac{\partial}{\partial x}(\rho_{t}%
\omega_{t})=\frac{\partial}{\partial x}\left[  \rho_{t}\left(  \frac
{\partial_{x}(\kappa\rho_{t}+\mu_{t})}{\rho_{t}}+\alpha x\right)  \right]  ,
\label{sol1}%
\end{equation}
satisfying $\rho(t)\rightarrow\rho_{0}$ as $t\rightarrow0^{+},$ $\rho_{t}\in
L_{\mathrm{{loc}}}^{1}((0,+\infty);W_{\mathrm{{loc}}}^{1,1}(\mathbb{R})),$
and
\[
\left\Vert \frac{\partial}{\partial x}\left(  \frac{\partial_{x}(\kappa
\rho_{t}+\mu_{t})}{\rho_{t}}+\alpha x\right)  \right\Vert _{L^{2}(\rho
_{t};\mathbb{R})}\in L_{\mathrm{{loc}}}^{2}(0,+\infty).
\]
\end{thm}

\textbf{Proof.} Since $\omega_{t}=\partial^{\circ}E_{\kappa,\alpha}(\rho_{t})$
and $\rho_{t}$ is the gradient flow of $E_{\kappa,\alpha}$, it follows from
Lemma \ref{lem1} that (\ref{sol1}) is satisfied by $\rho_{t}$ with the
additional conditions found in the statement of the theorem.

Now, observe that $\rho_{t}$ satisfies (\ref{sol1}) is equivalent to $\rho
_{t}$ satisfies
\begin{align*}
\frac{d}{dt}\int_{\mathbb{R}}\varphi(x)d\rho_{t}  &  =-\int_{\mathbb{R}%
}\varphi^{\prime}(x)d(\partial_{x}(\kappa\rho_{t}+\mu_{t}))-\alpha
\int_{\mathbb{R}}\varphi^{\prime}(x)xd\rho_{t}\\
&  =\int_{\mathbb{R}}\varphi^{\prime\prime}(x)d(\kappa\rho_{t}+\mu_{t}%
)-\alpha\int_{\mathbb{R}}\varphi^{\prime}(x)xd\rho_{t}\\
&  =\kappa\int_{\mathbb{R}}\varphi^{\prime\prime}(x)d\rho_{t}-\alpha
\int_{\mathbb{R}}\varphi^{\prime}(x)xd\rho_{t}+\int_{\mathbb{R}}%
\varphi^{\prime\prime}(x)d\mu_{t}\\
&  =\kappa\int_{\mathbb{R}}\varphi^{\prime\prime}(x)d\rho_{t}-\alpha
\int_{\mathbb{R}}\varphi^{\prime}(x)xd\rho_{t}+\frac{1}{2}\int_{\mathbb{R}%
^{2}}\frac{\varphi^{\prime}(x)-\varphi^{\prime}(y)}{x-y}d(\rho_{t}\times
\rho_{t}),
\end{align*}
in the sense of distributions on $t\in(0,\infty)$ and for all $\varphi\in
C_{0}^{\infty}(\mathbb{R)}.$\hfill$\diamond$\vskip12pt

\begin{remark}
Although in the general theory in \cite{Ambrosio} the previous
result is a characterization of gradient flow solutions, we do not
know how to get the converse in the characterization of the
element of minimal norm in Lemma {\rm\ref{lem1}} since we do not
how to show that $\mu$ is absolutely continuous with respect to
$\rho$. This implies that we do not know how to show that
distributional solutions with the properties written in Theorem
{\rm\ref{dist}} are gradient flow solutions.
\end{remark}

\begin{remark}
(Power-law potential) In the case of potential $W(x)=\left\vert
x\right\vert^{-\beta}$ with $0<\beta<1,$ we also have that the
associated gradient flows (see Remark {\rm\ref{Power}}) is a
solution in sense of distributions as in \eqref{weak1} or
\eqref{sol1}. Instead of \eqref{weak-L}, in this time the operator
$L_{\rho}(\varphi)$ is given by
\begin{align*}
L_{\rho}(\varphi) &  =\lim_{\delta\rightarrow0^{+}}\int_{|x-y|\geq\delta
}-\beta\frac{(x-y)}{\left\vert x-y\right\vert ^{\beta+2}}\varphi
(x)d(\rho\times\rho)\\
&  =\lim_{\delta\rightarrow0^{+}}\frac{1}{2}\int_{|x-y|\geq\delta}-\beta
\frac{1}{\left\vert x-y\right\vert ^{\beta}}\frac{\varphi(x)-\varphi(y)}%
{x-y}d(\rho\times\rho)<\infty,
\end{align*}
for $\varphi\in C_{0}^{1}(\mathbb{R)}.$ Thus, since $\rho\in D(E_{\beta,\kappa
,\alpha}),$
\[
\left\vert L_{\rho}(\varphi)\right\vert \leq\frac{\beta}{2}\left\Vert
\varphi^{\prime}\right\Vert _{\infty}\int_{\mathbb{R}^{2}}\frac{1}{\left\vert
x-y\right\vert ^{\beta}}d(\rho\times\rho)=C\mathcal{W}[\rho]\left\Vert
\varphi^{\prime}\right\Vert _{\infty}\text{ ,}%
\]
for all $\varphi\in C_{0}^{1}(\mathbb{R)}.$ Therefore $L_{\rho}\in(C_{0}%
^{1}(\mathbb{R))}^{\ast}$ and, similarly to \eqref{aux-repres},
$L_{\rho}(\varphi)=\int_{\mathbb{R}}\varphi^{\prime}d\mu,\ $for $\varphi\in C_{0}%
^{1}(\mathbb{R)}.$
\end{remark}

\end{document}